\documentclass[12pt]{amsart}
\usepackage[margin=1in]{geometry}
\usepackage{amsfonts}
\usepackage{amsthm}
\usepackage{thmtools}
\usepackage{amsmath}
\usepackage{amssymb}
\usepackage{graphicx}
\usepackage{tikz-cd}
\usepackage{mathrsfs}
\usepackage{array}
\usepackage{amssymb}
\usepackage{mathabx}
\usepackage{hyperref}
\hypersetup{
    colorlinks=true,
    linkcolor=blue,
    filecolor=green,      
    urlcolor=cyan,
    citecolor=magenta,
    pdftitle={The Monodromy of Compact Lagrangian Fibrations},
    pdfpagemode=FullScreen,
}
\usepackage{cleveref}

\newcommand{\R}{\mathbb{R}}
\newcommand{\Q}{\mathbb{Q}}

\newcommand{\C}{\mathbb{C}}
\newcommand{\Z}{\mathbb{Z}}
\newcommand{\End}{\text{End}}

\newcommand{\Hom}{\text{Hom}}

\newcommand{\GL}{\text{GL}}

\newcommand{\rightQ}[2]{\left.\raisebox{.2em}{$#1$}\middle/\raisebox{-.2em}{$#2$}\right.}
\newcommand{\leftQ}[2]{\left.\raisebox{-.2em}{$#2$}\middle\backslash\raisebox{.2em}{$#1$}\right.}

\newtheorem{theorem}{Theorem}[section]
\newtheorem{proposition}[theorem]{Proposition}
\newtheorem{lemma}[theorem]{Lemma}
\newtheorem{corollary}[theorem]{Corollary}

\theoremstyle{definition}
\newtheorem{definition}[theorem]{Definition}
\newtheorem{example}[theorem]{Example}

\theoremstyle{remark}
\newtheorem{remark}[theorem]{Remark}

\numberwithin{equation}{section}

\title{The monodromy of compact Lagrangian fibrations}
\author{Edward Varvak}
\address{Edward Varvak: Department of MSCS, University of Illinois Chicago, 
Chicago, USA}
\email{evarva2@uic.edu}

\begin{document}

\begin{abstract}
    We study the monodromy representations underlying compact Lagrangian fibrations. 
    In the case where the associated period map is generically immersive, we prove 
    that the mondromy representation is irreducible over \(\mathbb{C}\). In the
    alternative case where the fibration is isotrivial, we recover a result of
    \cite{kim-laza-martin23}, proving that its fibers are isogeneous to a power of
    an elliptic curve. We show that over \(\mathbb{C}\), the monodromy representation
    underlying an isotrivial Lagrangian fibration is a direct sum of two irreducible
    \(\mathbb{C}\)-local systems.
\end{abstract}

\maketitle

\section{Introduction}

The Beauville--Bogomolov decomposition theorem states that any compact K\"ahler
manifold \(X\) with \(K_X\) trivial is up to isogeny a product of complex tori,
strict Calabi--Yau manifolds, and hyperk\"ahler manifolds. In recent years, the 
latter of the three have become increasingly popular. For a survey, see for instance
\cite{deb22}.

\begin{definition}
    Let \(X\) be a compact simply-connected K\"ahler manifold. We say that \(X\) is
    hyperk\"ahler if \(H^0(X, \Omega_X^2) = \C \sigma\) for some nowhere
    degenerate holomorphic 2-form \(\sigma\).
\end{definition}

Since \(\sigma\) is symplectic, \(X\) must have even complex dimension \(2n\).
We may hope to understand the geometry of a hyperk\"ahler by expressing it as the
total space of a fibration. To that end, consider the following definition.

\begin{definition}
    Let \(f: X \to B\) be a holomorphic map from a hyperk\"ahler to a normal analytic
    variety of dimension \(0 < \dim B < 2n\).
    We say that \(f\) is a fibration if it is proper and surjective with connected 
    fibers. If for some fiber \(X_b \subseteq X\) we have \(\sigma|_{X_b} = 0\),
    we say that \(X_b\) is isotropic. If every fiber of \(f\) is Lagrangian 
    (isotropic and pure \(n\)-dimensional) we say that \(f\) is a Lagrangian fibration.
\end{definition}

While the condition of being Lagrangian may seem quite strict, it turns out to be 
forced, by \cite{mats99}. In fact, if \(f: X \to B\) is a fibration, then
\(f\) is Lagrangian and \(B\) is projective; see, for instance,
\cite[Proposition 2.2]{greb-lenh14}. Of special note 
are the smooth fibers of \(f\); by \cite{mats99}, they are all Abelian varieties.
Additionally, Matsushita establishes many constraints on the base of a Lagrangian 
fibration. In particular, \(B\) is \(\Q\)-Fano, \(\Q\)-factorial, and log terminal
\cite{mats15}. \(B\) has Picard rank 1, and has the same 
\(\Q\)-intersection cohomology as \(\mathbb{P}^n\). This fact holds even on the 
level of Hodge decompositions, in that every non-zero cohomology class on \(B\)
is Hodge. A central conjecture of Matsushita predicts that \(B \simeq \mathbb{P}^n\) for every
Lagrangian fibration \(f: X \to B\). This result is known to be true in several cases,
including when \(B\) is a curve or a surface \cite{huy-xu22}. In fact, Hwang proves
in \cite{hwa08} that \(B \simeq \mathbb{P}^n\) whenever \(B\) is smooth.

We denote by \(B^\circ \subseteq B\) the complement of the discriminant locus and 
abuse notation to denote by \(f: X^\circ := X_{B^\circ} \to B^\circ\) the smooth 
fibration. We have an associated
period map \(B^\circ \to \leftQ{\mathcal{D}}{\Gamma}\), mapping each Abelian variety
to its associated weight one Hodge structure. By Torelli, the smooth fibers of \(f\)
are determined up to isogeny by the associated polarized variation of Hodge structures
\(V_\Q = R^1 f_* \Q_{X^\circ}\). Thus, a common approch for understanding the geometry
of \(f\) is understanding the structure of local systems underlying variations of 
Hodge structures on \(B\). \cite{geeman-voisin16} and \cite{bak22} affirm a conjecture
of Matsushita, proving that the period map
associated to \(f\) is either constant or generically immersive. In the first case,
we say that \(f\) is \textbf{isotrivial}; in the second, \(f\) is
\textbf{of maximal variation}. Voisin proves that \(V_\R = V_\Q \otimes_\Q \R\) is
irreducible as a variation of Hodge structures \cite{vois18}. To further strengthen
this result, we consider the isotrivial and maximal variation cases separately.

\begin{theorem}[=\Cref{complex-irreducibility}]
    Suppose that \(V_\Q\) is the local system associated to the variation of Hodge
    structures arising from a maximal variation Lagrangian fibration. Then
    \(V_\C = V_\Q \otimes_\Q \C\) is irreducible as a representation of
    \(\pi_1(B^\circ)\).
\end{theorem}

As will be evident from the proof of the previous theorem, the assumption that 
\(f\) is of maximal variation is critical. Not only is there no analogous result
in the isotrivial case, we prove that the \(\C\)-local system underlying 
an isotrivial Lagrangian fibration is never irreducible! 
The isotrivial case was studied in depth by Kim, Laza, and Martin
\cite{kim-laza-martin23}; we recover their structure theorem
\cite[Theorem 1.3]{kim-laza-martin23}.

\begin{theorem}[=\Cref{elliptic-curve-theorem}] \label{intro-elliptic-curve-theorem}
    Suppose that \(f: X \to B\) is isotrivial. Then there is an elliptic curve \(E\)
    such that \(E^n\) is isogenous to the smooth fiber \(X_b\) of \(f\).
\end{theorem}

This allows us to understand the splitting behaviour of \(V_\C\) into irreducible 
\(\C\)-local systems. It turns out that any splitting of \(V_\C\) occurs
over a finite extension field \(K\) of \(\Q\), which we identify with the CM field
of the associated elliptic curve.

\begin{theorem}[=\Cref{isotrivial-rep-classification}] 
    Let \(f: X \to B\) be an isotrivial Lagrangian fibration, and let \(E\) be an 
    elliptic curve for which \(X_b\) is isogeneous to \(E^n\) as in \Cref{intro-elliptic-curve-theorem}.
    Denote by \(K = \mathrm{End}_{\Q HS}(H^1(E, \Q))\) the CM field of \(E\). There is
    a splitting of \(V_K\) into
    \(K\)-local systems of rank \(n\), \(V_K = U_1 \oplus U_2\), such that
    \(U_1\) and \(U_2\) are irreducible over \(\C\). Furthermore, 
    \(U_1 \simeq U_2\) if and only if they are defined over \(\Q\).
\end{theorem}

Finally, we relate our results to an alternative classification of isotrivial 
Lagrangian fibrations found in \cite{kim-laza-martin23} to produce a more geometric
description of the irreducible isotrivial \(V_\C\). This extends the discussion in
\cite[Remark 1.8]{kim-laza-martin23} relating the CM structure of the 
elliptic curve \(E\) to a trivializing cover of the local system \(V_\C\).

\begin{proposition}[=\Cref{application-isotrivial-real}]
    Suppose that \(V_\C\) is an irreducible isotrivial complex variation of
    polarized Hodge structures associated to \(f: X \to B\). Then
    a minimal trivializing cover \(C^\circ \to B^\circ\) of \(V_\C\) 
    compactifies to an Abelian torsor \(A \supseteq C^\circ\).
\end{proposition}

\begin{remark}
Suppose further that \(f\) admits a rational section, and that for
every Lagrangian fibration \(X' \dashrightarrow B'\) birational to \(f\)
the conjecture of Matsushita holds (i.e. \(B' \cong \mathbb{P}^n\)). Then
as an immediate consequence of \cite[Theorem 1.8]{kim-laza-martin23},
\(f\) is birational to a \(\text{K3}^{[n]}\)-fibration or a 
\(\text{Kum}_n\)-fibration.
\end{remark}

\subsection*{Acknowledgements}
I thank my advisor Ben Bakker for suggesting this approch, as well as his 
long-standing patience and support throughout many helpful conversations.

\subsection*{Notation and conventions}
Throughout this paper, we mainly stick to the following notations:
\begin{itemize}
\item \(X\) a (compact) hyperk\"ahler manifold of dimension \(2n\) with symplectic
form \(\sigma\)
\item \(f: X \to B\) a Lagrangian fibration
\item \(B^{sm}\) the smooth locus of \(B\), \(B^\circ \subseteq B^{sm}\) the 
regular values of \(f\)
\item \(j: B^{sm} \to B\) the natural inclusion, 
\(\Omega_B^{[k]} = j_* \Omega_{B^{sm}}^k\) the sheaf of reflexive differentials
\item \(f: X^\circ \to B^\circ\) the restriction of \(f\) to the smooth fibers
\item \(V_\Q\) the local system \(R^1 f_* \Q_{X^\circ}\) on \(B^\circ\)
\item \((V_\Q, (\mathcal{V}, \nabla), F^\bullet \mathcal{V}, q)\) the associated 
variation of Hodge structures, where
\item \(\mathcal{V} = V \otimes_\Q \mathcal{O}_{B^{\circ}}\) is a bundle equipped 
with a flat connection \(\nabla\)
\item \(F^\bullet\) a decreasing filtration of \(\mathcal{O}_{B^\circ}\)-modules on
\(\mathcal{V}\)
\item \(q: V_\Q \otimes V_\Q \to \Q_{B^\circ}(-1)\) a polarization form
\item \(\mathcal{V}^{>-1}\) Deligne's canonical extension of \(\mathcal{V}\) 
to \(B^{sm}\), with residue eigenvalues in \((-1, 0]\)
\item \(V_K = V_\Q \otimes_\Q K\) for \(\Q \subseteq K \subseteq \C\) a field 
extension
\item \(U \subseteq V_\C\) an irreducible \(\C\)-local system
\item \(W \subseteq V_\R\) an \(\R\)-local system
\end{itemize}

\section{Structure Theory of Complex Representations}

Before we begin our discussion of complex representations arising from real 
representations, we review complex variations of Hodge structures. As with 
rational or real variations, we have two descending filtrations \(F^\bullet\)
and \(\overline{F}^\bullet\) on a \(\C\)-local system \(V\), which are holomorphic and 
anti-holomorphic respectively. For a clean summary, see \cite{bak22}. The key
difference with the more conventional notions of variations of Hodge structure 
comes from the lack of bigrading; since there is no longer a clear notion of 
complex conjugation, \(V\) only admits a (mono)grading 
\(V^p := F^p V \cap \overline{F}^{-p} V\).

The key invariant for complex variations is the level; this is the 
difference between the maximum and minimum \(p\) for which \(V^p \neq 0\). For
example, \(\C(-p)\) is level zero, while the variation associated to a family of 
ellipic curves is level one. Notice that if \(V\) is level \(l\) and 
\(V'\) is level \(l'\), the tensor product \(V \otimes V'\) is level \(l + l'\).

Drawing upon the theorem of the fixed part \cite{schmid73}, Deligne proves the 
following crucial structure theorem for complex variations of Hodge structure.

\begin{proposition}[\cite{del87}, Proposition 1.13]\label{deligne-structure-theorem}
    Suppose $V$ is a $\C$-local system underlying a polarized variation of Hodge 
    structures. We have a splitting of $\C$-local systems
    \begin{equation}
        \label{prime-factorization}
        V = \bigoplus_i B_i \otimes U_i
    \end{equation}
    where the $U_i$ are irreducible and pairwise non-isomorphic $\C$-local systems,
    and the $B_i$ are $\C$-vector spaces. Moreover,
    \begin{itemize}
        \item[(i)] Each $U_i$ underlies a polarized complex variation of Hodge 
        structures, unique up to shifting the Hodge grading, and
        \item[(ii)] Each polarizable complex variation of Hodge structure of \(V\) 
        arises from \(\mathrm{(i)}\) by equipping each $U_i$ with its unique (up to shifts) polarizable complex
        variation of Hodge structures, and equipping each $B_i$ with a polarizable
        complex Hodge structure coming from \cite{schmid73}, via
        $B_i = \mathrm{Hom}(V, U_i)$.
    \end{itemize}
\end{proposition}

The subsequent observation made in \cite{bak22} is that if $V$ is a $\C$-local system 
equipped with a polarized complex variation of Hodge structures, then
$V =  W_\C$ for some polarized real variation of Hodge structures $W$ 
if and only if complex conjugation of $V$ flips the Hodge grading. Thus, every
real local system arising from a real polarizable variation of Hodge structure is
semi-simple.

This allows us to draw on the following classification theorem, which is a classical
result of representation theory. Despite the fact that this theorem is well known,
we could not find a satisfactory reference for representations of (non-finite) groups,
so we have included a proof below.
\begin{proposition} \label{rep-classification}
    If $V_\R$ is an irreducible real representation underlying a polarizable 
    variation of Hodge structures, then the isotypic factors of 
    $V_\C$ coming from \Cref{deligne-structure-theorem} have the following
    distinct forms:
    \begin{itemize}
    \item[\textbf{(R)}]  $B \otimes W_\C$, for $W$ an irreducible 
    $\R$-local system,
    \item[\textbf{(C)}] 
    $(B \otimes U) \oplus (\overline{B} \otimes \overline{U})$ for
    $U$ an irreducible $\C$-local system,
    \item[\textbf{(Q)}]  $B \otimes U$, where $U^{\oplus 2} = W_\C$
    for $W$ an irreducible $\R$-local system.
\end{itemize}
\end{proposition}

In the first case, we say that \(V_\C\) is of \textbf{real type}; in the second, 
that \(V_\C\) is of \textbf{complex type}; and in the third case, that \(V_\C\) is of 
\textbf{quaternionic type}.

\begin{proof}
Suppose that $U$ is a complex irreducible factor as above. If $U \neq \overline{U}$,
then somewhere in the factorization of $V_\C$ there is a factor $\overline{U}$,
putting us in the second case. Thus, we may suppose that $U = \overline{U}$. The
restriction of the polarization form to $U$ gives a non-degenerate pairing
\( q: U \times U \to \C(-d) \), inducing an isomorphism of $\C$-vector spaces
\begin{equation}
    \varphi: U \to U^\vee, \quad z \mapsto q(-, z).
\end{equation}
We may associate to $\varphi$ the maps $\varphi^{\textrm{sym}} = \varphi + \varphi^T$
and $\varphi^{\textrm{alt}} = \varphi - \varphi^T$, for \(\varphi^T\) the transpose.
If both maps are non-zero, observe that they are linearly independent. Now if we take 
\(h: U \times \overline{U} \to \C(-d)\) to be the Hermitian form associated to the
polarization \(q\), we get a similar isomorphism 
\begin{equation}
    \psi: \overline{U} \to U^\vee, \quad \overline{z} \mapsto h(-, \overline{z}).
\end{equation}
Since $U = \overline{U}$, we see that $\varphi$, $\varphi^{\textrm{sym}}$, and
$\varphi^{\textrm{alt}}$ live in $\End(U)$. By Schur's lemma, they are therefore
given by multiplication by a complex number. In particular, the latter two are
linearly dependent, a contradiction. Thus, one of them is zero. This shows that
$\varphi$ is either symmetric or alternating.

Next, consider the following diagram:
\begin{equation*}
    \begin{tikzcd}
    U \arrow[dashed]{d}{\sigma} \arrow{rr}{\varphi} && U^\vee \arrow[equals]{d} \\
    \overline{U} \arrow{rr}{\psi} && U^\vee
    \end{tikzcd}
\end{equation*}

We may uniquely define 
\(\sigma = \psi^{-1} \circ \varphi: U \to \overline{U}\)
an anti-linear map relating the polarization form $q$ with the Hermitian form $h$:
\begin{equation}
    q(x, \sigma(y)) = h(x, y).
\end{equation}
Notice that $\sigma^2: U \to U$ is $\C$-linear, so again by Schur's lemma we may
express $\sigma^2 = \lambda \cdot \mathrm{id}$ for some non-zero $\lambda \in \C$.
Applying this in conjunction with our observation that $\varphi$ is either symmetric
or alternating, we get
\begin{align*}
    h(\sigma(x), \sigma(x)) &= q(\sigma(x), \sigma^2(x)) \\
    &= q(\sigma(x), \lambda x) \\
    &= \lambda q(\sigma(x), x)) \\
    &= \pm \lambda q(x, \sigma(x)) \\
    &= \pm \lambda h(x, x).
\end{align*}
But because $h$ is positive definite, we see that $\lambda$ is a positive real
precisely when $\varphi$ is symmetric, and a negative real when $\varphi$ is
alternating.

Now suppose that $\lambda > 0$. Rescaling $\sigma$ by $\lambda^{-\frac{1}{2}}$, we get
$\sigma^2 = \mathrm{id}$. Let $\text{Res}_{\C/\R}(U)$ be the underlying
$\R$-vector space, and consider the induced map 
\((\sigma - \mathrm{id}): \text{Res}_{\C/\R}(U) \to \text{Res}_{\C/\R}(U)\).
This map has a half-dimensional kernel, namely the 1-eigenspace of \(\sigma\), which
we denote $W$. Noting that $W_\C = U$ shows that $W$ must be irreducible,
placing us in the first case. Alternatively, suppose that $\lambda < 0$.
Rescaling $h$ by $(-\lambda)^{-\frac{1}{2}}$,
we get $\sigma^2 = -\mathrm{id}$. We can see that $U$ has no underlying real structure:
indeed, such a structure comes with a symmetric form and, as we have shown above, all forms
on $U$ are alternating. However, we may define
\begin{equation}
    \Sigma = \begin{pmatrix}
        0 & \sigma \\
        -\sigma & 0
    \end{pmatrix}: U^{\oplus 2} \to U^{\oplus 2}.
\end{equation}
Then we once again have $\Sigma^2 = \textrm{id}$, so as above, we take $W$ to be the
1-eigenspace of $\text{Res}_{\C/\R}(U^{\oplus 2})$. We must now argue that $W$ is
irreducible as an $\R$-representation. Write \(W = W' \oplus W''\) for some 
\(\R\)-subrepresentations \(W', W'' \subseteq W\). We will argue that one of 
\(W', W''\) is zero. Tensor by \(\C\) to get \(U^{\oplus 2} = (W')_{\C} \oplus (W'')_{\C}\).
$U$ is irreducible, so either $(W')_\C = (W'')_\C = U$, or one of the factors
is zero. The former is impossible, as that would imply that $U$ has an underlying
real structure $W'$. This proves that $W$ is irreducible, putting us in the third case.
\end{proof}

Notice from  the proof of \Cref{rep-classification} that the crucial object to 
consider is \(\bigwedge^2 V_\C\). As a local system, its global sections probe for
irreducible complex subrepresentations of every type, as we will see later. 
In particular, we have the following corollary.

\begin{corollary}\label{proof-of-rep-classification}
    Let \(U\) be an irreducible complex variation of Hodge structures with 
    \(U \simeq \overline{U}\). Then \(h^0(\mathrm{Sym}^2 U) = \C\varphi \neq 0\) 
    if and only if \(U\) is of real type, and \(h^0(\bigwedge^2 U) = \C \varphi \neq 0\)
    if and only if \(U\) is of quaternionic type.
\end{corollary}

Voisin proves another very useful theorem regarding the irreducibility of \(V_\R\)
by studying the global sections of \(\bigwedge^2 V_\R\).

\begin{lemma}[\cite{vois18}, Lemma 4.5] \label{real-irreducibility}
    Let \(f: X \to B\) be a Lagrangian fibration, and \(V_\R = R^1 f_* \R_{X^\circ}\).
    $V_\R$ is irreducible as a polarizable variation of real Hodge structures.
\end{lemma}
\begin{proof}
    By \cite{mats16}, the restriction $H^2(X, \R) \to H^2(X_b, \R)$ is
    rank 1 for $X_b$ a generic fiber. But by Deligne's global invariant cycles theorem
    \cite{del71}, the map
    \begin{equation}
        \label{deligne-inv-cycles}
        H^2(X, \R) \to H^0(B^\circ, R^2 f_* \R_{X^\circ})
    \end{equation}
    is surjective. In particular, the dimension of the codomain of
    \eqref{deligne-inv-cycles} is at most 1. Note that
    $R^2 f_* \R_{X^\circ} = \bigwedge^2 V_\R$, so if $V_\R$ split as a variation,
    the space of sections of $\bigwedge^2 V_\R$ would be greater than one-dimensional,
    a contradiction.
\end{proof}

\begin{remark}
Throughout this paper, \(X\) and \(B\) are assumed to be compact. This assumption 
is not strictly necessary for many of our tools 
(i.e. \Cref{deligne-structure-theorem}, \Cref{proof-of-rep-classification},
\Cref{deligne-ext}), but \Cref{real-irreducibility} is not guaranteed when \(B\)
is non-compact. 
\end{remark}

\section{Review of Algebraic Foliations}

We closely follow the survey of Araujo \cite{araujo18}, taking their conventions.

\begin{definition}
    A foliation \(\mathcal{E}\) on a normal variety \(B\) is a saturated coherent
    subsheaf of the tangent sheaf which is closed under the ambient Lie bracket.
    The canonical class \(K_\mathcal{E}\) is any Weil divisor for which
    \(\det(\mathcal{E}) \simeq \mathcal{O}_B(-K_\mathcal{E})\).
    We say that \(\mathcal{E}\) is (\(\Q\)-)Gorenstein if \(K_\mathcal{E}\) is 
    (\(\Q\)-)Cartier. If \(-K_{\mathcal{E}}\) is ample, we say that \(\mathcal{E}\)
    is (\(\Q\)-)Fano.
\end{definition}

If we denote the (generic) rank of \(\mathcal{E}\) by \(r\), we may consider
the so-called Pfaff field  (\cite[Definition 2.2]{araujo18})
\begin{equation} \label{eq-pfaff-field}
    \Omega_B^r \to \Omega_B^{[r]} \to \det(\mathcal{E})^\vee \simeq 
    \mathcal{O}_B(K_{\mathcal{E}}).
\end{equation}
Here and throughout this paper the notation \(\Omega_B^{[r]}\) denotes the
reflexivization \(\left(\bigwedge^r T_B\right)^\vee\). We define the singular locus 
of \(\mathcal{E}\) to be the closed subscheme of \(B\) whose ideal sheaf is the image
of the associated map \(\left(\Omega_B^r \otimes \det(\mathcal{E})\right)^{\vee \vee}
\to \mathcal{O}_B\). On the smooth locus of \(B\), this is simply the subvariety
where \(\mathcal{E}\) is not a bundle.

\begin{definition}
    An analytic subvariety \(Z \subseteq B\) not contained in the singular locus of 
    \(\mathcal{E}\) is called invariant if on the smooth locus of \(Z\), 
    \eqref{eq-pfaff-field} factors through the induced map on tangent bundles as
\begin{equation*}
    \begin{tikzcd}
    \Omega_B^r|_{Z^{sm}} \arrow{r} \arrow{rd} & \Omega_Z^r|_{Z^{sm}} \arrow{d} \\
    & \mathcal{O}_B(K_\mathcal{E})|_{Z^{sm}}
    \end{tikzcd}
\end{equation*}
A maximal invariant subvariety of dimension \(r\) is called a leaf of \(\mathcal{E}\).
\end{definition}

On the regular locus of the foliation, a generalization
of the classical Frobenius theorem ensures the existence of smooth submanifolds
whose tangent bundles recover \(\mathcal{E}\). For details, see for instance 
\cite[Proposition 1.3.3]{bog-mcq16}. However, despite the fact that \(\mathcal{E}\)
is an algebraic sheaf, it is generally false that the leaves of \(\mathcal{E}\)
are themselves algebraic. To this end, Bogomolov and McQuillan give a sufficient
condition to ensure the existence of rationally connected algebraic leaves.

\begin{proposition}[\cite{bog-mcq16}, Main Theorem] \label{bog-mcq-main-theorem}
    Let $\mathcal{E}$ be a foliation on a normal projective variety $B$,
    and let $C \subseteq B$ be a curve on which $\mathcal{E}|_C$ is ample. Then for
    any $p \in C$, the leaf $Z \ni p$ is an algebraic variety. The leaf through 
    a general point of \(C\) is moreover rationally connected.
\end{proposition}

The difficult part of the theorem is proving that the leaf \(Z\) is indeed rationally
connected. We do not need this part of the theorem, so we omit its proof. Instead,
we will sketch their argument to show that the leaf \(Z\) is algebraic.

\begin{proof}
To begin, notice that over the curve \(C\), the foliation has locally the structure
of a fibration \(Z_\Delta \to \Delta \subseteq C\). It is possible that when we try
to extend this structure globally, the leaves \(Z_\alpha\) intersect \(C\) at
more than one point. Thus, we instead consider the graph \(\Gamma\) of the inclusion
\(C \to B\). This sits as a curve inside the product \(C \times B\), and
\(\mathcal{E}\) naturally pulls back to a foliation on this product, with leaves
through \(\Gamma\) lying over the leaves of \(\mathcal{E}\) through \(C\).
Notice further that this construction has the effect of separating the leaves:
two leaves through \(\Gamma\) map to the same leaf \(Z \subseteq B\) if and only
if \(Z\) meets \(C\) twice. Thus we produce a fibration \(W \to \Gamma\).

We aim to show that the total space \(W\) is algebraic. This will show that each of
the individual fibers is algebraic, thereby making its images in \(B\) algebraic.
Let \(r\) be the rank of \(\mathcal{E}\); we need to show that the Zariski closure
\(\overline{W}\) of \(W\) has the expected dimension \(r+1\).
We proceed asymtotically: if \(\mathcal{L}\) is any line bundle on \(\overline{W}\), 
it suffices to show that
\begin{equation} \label{eq-asymtotic-bound}
    h^0(\overline{W}, \mathcal{L}^{\otimes n}) \leq C(\mathcal{L}) \cdot n^{r+1},
\end{equation}
for \(C(\mathcal{L})\) a constant depending only on \(\mathcal{L}\).

Take \(\widehat{W}\) to be the completion of \(W\) along \(\Gamma\).
The inclusion \(H^0(\overline{W}, \mathcal{L}) \hookrightarrow 
H^0(\widehat{W}, \mathcal{L})\) tells us that it is sufficient to show
\eqref{eq-asymtotic-bound} for \(\widehat{W}\). If \(W_m\) denotes the \(m\)-th
infinitesimal thickening of \(\overline{W}\), we have the exact sequence
\begin{equation} \label{eq-principal-parts}
    0 \to H^0(\Gamma, \text{Sym}^m N_{\Gamma / W}^\vee 
        \otimes \mathcal{L}^{\otimes n}) \to 
    H^0(W_{m+1}, \mathcal{L}^{\otimes n}) \to 
    H^0(W_{m}, \mathcal{L}^{\otimes n}).
\end{equation}
Here \(N_{\Gamma / W}\) is the normal bundle of \(\Gamma\). But notice that 
by construction, it admits a natural map \(\mathcal{F} \to N_{\Gamma / W}\)
from the pullback of \(\mathcal{E}\), which is generically an isomorphism.
By assumption, \(\mathcal{F}\) is positive, so there is some
\(m = C(\mathcal{L})' \cdot n\) above which the left term of
\eqref{eq-principal-parts} vanishes. Thus
\begin{align*}
    h^0(\widehat{W}, \mathcal{L}^{\otimes n}) 
    &= h^0(W_m, \mathcal{L}^{\otimes n}) \\
    &\leq \sum_{k=0}^m h^0(\Gamma, \text{Sym}^k N_{\Gamma / W}^\vee 
        \otimes \mathcal{L}^{\otimes n}) \\
    &\leq \sum_{k=0}^m C(\mathcal{L})'' \cdot n^r \\
    &= C(\mathcal{L})' \cdot C(\mathcal{L})'' \cdot n^{r+1}
\end{align*}
Setting \(C(\mathcal{L}) = C(\mathcal{L})' \cdot C(\mathcal{L})''\) gives the
required bound.
\end{proof}

\section{Fibrations of Maximal Variation}

We first begin by refining \Cref{deligne-structure-theorem} under the assumption
that $V_\Q$ is of maximal variation. Write $V_\C = \bigoplus_i B_i \otimes U_i$.
We claim that in fact, there is more than one isotypic component in this sum if and
only if it is of the complex form
$V_\C = (B \otimes U) \oplus (\overline{B} \otimes \overline{U})$. 
Indeed, \(V_\C\) comes from a real representation, so any isotypic component
of \(V_\C\) has a complex conjugate (possibly itself) also contained in \(V_\C\).
We may add any isotypic component to its complex conjugate to produce
$W_\C \subseteq V_\C$; this must arise from a real variation of Hodge structure
\(W \subseteq V_\R\). But \(V_\R\) is irreducible as a variation by 
\Cref{real-irreducibility}, so we get $W_\C = V_\C$ as needed.

In fact, we can say much more about the decomposition of \(V_\C\) into irreducibles.

\begin{lemma} \label{mv-rep-classification}
    Let $V_\C$ be a variation of Hodge structures arising from a maximal variation
    Lagrangian fibration. Then for some irreducible complex variation of Hodge
    structures $U$,
    \begin{equation*}
        V_\C = \begin{cases}
        U & \text{if } V_\C \text{ has real type}\\
        U \oplus \overline{U} & \text{if } V_\C \text{ has complex type}\\
        U \oplus U & \text{if } V_\C \text{ has quaternionic type}
    \end{cases}.
    \end{equation*}
\end{lemma}

\begin{proof}
Recall that $V_\Q$ is a variation of weight one Hodge structures. Since it is of
maximal variation, the irreducible factor \(U \subseteq V_\C\) has level one by 
\cite[Lemma 4]{bak22}. Observe additionally that \(V_\C\) has only a
\(V^{1,0}\) piece and a \(V^{0,1}\) piece, so it too is of level one.
But now notice that this severely limits the level of $B$.
Indeed, if $U$ and $V_\C = B \otimes U$ are both level one, then $B$ is of level zero.
This means that any vector subspace \(C \subseteq B\) is a Hodge subvariation.
Concretely, if \(V_\C\) has real type, \(V_\C = B \otimes W_\C\), then any 
one-dimensional subspace \(\C(0) \subseteq B\) yields a Hodge subvariation
\(\C(0) \otimes W_\C \subseteq V_\C\) coming from a real variation 
\(\R(0) \otimes W \subseteq V_\R\). But by \Cref{real-irreducibility}, 
\(W = V_\R\), so that \(B\) is one-dimensional.

The complex case is completely analogous; if
\(V_\C = (B \otimes U) \oplus (\overline{B} \otimes \overline{U})\), then 
both \(U\) and \(\overline{U}\) are level one, forcing \(B\) and \(\overline{B}\)
to be level zero. But now \(\C(0) \subseteq B\) and
\(\overline{\C}(0) \subseteq \overline{B}\) are Hodge substructures. This yields 
\(U \oplus \overline{U} \subseteq V_\C\), which is underlied by a real variation.
But again by \Cref{real-irreducibility}, the underlying real variation is \(V_\R\),
forcing \(\dim B = 1\).

Finally, when \(V_\C = B \otimes U\) is quaternionic, we know that \(V_\C\) 
is itself real, so that \(\dim B \geq 2\). Again, by the same argument, \(B\) 
has level zero, so any subspace of \(B\) yields a subvariation. This time, 
consider \(\C(0)^{\oplus 2} \subseteq B\). As before, this is a subvariation. 
Moreover, \(\C(0)^{\oplus 2} \otimes U \subseteq V_\C\) is a subvariation 
underlied by a real variation of Hodge structures. Again by 
\Cref{real-irreducibility}, we get equality \(B = \C(0)^{\oplus 2}\), 
so that \(V_\C \simeq U^{\oplus 2}\).
\end{proof}

With this description, we can compute some explicit examples. In particular,
we see that starting from an elliptic K3 surface, we can get a concrete understanding
of the associated \(\text{K3}^{[n]}\)-type Lagrangian fibration.

\begin{example}
Let \(X \to \mathbb{P}^1\) be an elliptic K3 surface
which is of maximal variation. By \Cref{mv-rep-classification}, the associated 
complex variation of Hodge structures \(V_\C = U\) is irreducible; otherwise, 
each of the irreducible factors would have to be one-dimensional, thereby failing 
to be of level one. Consider the associated \(\text{K3}^{[n]}\)-type Lagrangian 
fibration; let \((\mathbb{P}^n)^\circ\) be the complement of the discriminant of 
\(f^{[n]}\) in \(\mathbb{P}^n\).
Denote by \(\left(X^{[n]}\right)^\circ \to (\mathbb{P}^n)^\circ\)
the restriction of \(f^{[n]}\) to the preimage of \((\mathbb{P}^n)^\circ\).
Likewise, let \(U \subseteq \mathbb{P}^1\) be the complement of the discriminant of \(f\),
and \(X^\circ \to U\) be the restriction of \(f\) to \(f^{-1}(U)\).
We have a rational map \(U^n \to U^{(n)} \dashrightarrow \mathbb{P}^n\)
to the \(n\)-th symmetric power of \(U\); let
\((U^n)^\circ \subseteq (\mathbb{P}^1)^n\) be the locus on which it is regular. This 
gives us the diagram
\begin{equation*}
\begin{tikzcd}
	{(X^n)^\circ} \arrow{r} \arrow{d} & {\left(X^{[n]}\right)^\circ} \arrow[hook]{r} \arrow{d} & {X^{[n]}} \arrow{d} \\
	{(U^n)^\circ} \arrow{r} & {(\mathbb{P}^n)^\circ} \arrow[hook]{r} & {\mathbb{P}^n}
\end{tikzcd}
\end{equation*}

Notice that this yields a right-exact sequence of groups 
\begin{equation}
    \pi_1((U^n)^\circ) \to \pi_1((\mathbb{P}^n)^\circ) \to \mathfrak{S}_n \to 1.
\end{equation}
Take \(H = \pi_1((U^n)^\circ)\) and \(G = \pi_1((\mathbb{P}^n)^\circ)\), and 
consider the induction \(\text{Ind}_{G/H}(V_\C)\). Notice that 
by construction, this representation is expressable as \(B \otimes W\), where 
\(W\) is the representation associated to \(R^1 f_* \C_{X^\circ}\). Now notice that 
\(\sum_{p \in \mathfrak{S}_n} pW \) carries the structure of a 
\(G\)-subrepresentation of \(V_\C\). Moreover, \(\mathfrak{S}_n\) acts transitively 
on the irreducible factors of \(B \otimes W\), so if
\(\sum_{p \in \mathfrak{S}_n} pW \) split as a sum of distinct irreducible 
\(G\)-representations, there would have to be a non-zero map between them. 

Thus \(V_\C\) has a single isotypic component. For each irreducible factor \(W'\),
we may consider its restriction to \(H\). Since \(\mathfrak{S}_n\) acts transitively 
on the irreducible factors of \(B \otimes W\), there is some element of
\(\mathfrak{S}_n\) which acts by permuting the irreducible factors. But each \(W'\)
was a \(G\)-representation, so it must have been closed under the action of 
\(\mathfrak{S}_n\). Therefore \(W' = V_\C\) is a single irreducible factor. 
By \Cref{mv-rep-classification}, \(V_\C\) has real type.
In particular, we see that real type Lagrangian fibrations occur over projective
space of every dimension.
\end{example}

We will soon see every Lagrangian fibration of maximal variation is of real type.
Consider the map \(f: X \to B\) on three levels:
\begin{equation*}
\begin{tikzcd}
	{X^\circ} \arrow[hook]{r} \arrow{d} & {X_{B^{sm}}} \arrow[hook]{r} \arrow{d} & X \arrow{d}\\
	{B^\circ} \arrow[hook]{r}{i} & {B^{sm}} \arrow[hook]{r}{j} & B
\end{tikzcd}
\end{equation*}
Here \(B^{sm}\) is the non-singular locus of \(B\), and \(X_{B^{sm}} \to B^{sm}\) the 
pullback of \(f\). Likewise, \(B^\circ \subseteq B^{sm}\) is the complement of the
discriminant of \(X_{B^{sm}} \to B^{sm}\), and \(X^{\circ} \to B^\circ\) is the 
associated pullback.

Let $(\mathcal{V}, \nabla, F^{\bullet})$ be the flat bundle on \(B^\circ\)
associated to \(V\), equipped with the sub-bundle $F^1\mathcal{V}$. A theorem of 
Matsushita \cite[Theorem 1.2]{mats05} states that when \(B\) is smooth and 
\(X\) is projective, we have an isomorphism \(R^1 f_* \mathcal{O}_X \simeq \Omega_B^1\).
Schnell \cite[\S 4]{schnell23} vastly generalizes this result, showing that one only 
needs the total space to be holomorphic symplectic and K\"ahler. In particular,
\(R^1 f_* \mathcal{O}_{X_{B^{sm}}} \simeq \Omega_{B^{sm}}^1\).

If we further restrict to \(X^\circ\) and take duals, we get
\(T_{B^\circ} \simeq f_* \Omega_{X^\circ/B^\circ}^1 = F^1 \mathcal{V}.\)
In fact, Bakker and Schnell use Saito's theory of Hodge modules to extend this 
isomorphism over \(B^{sm}\).
\begin{lemma}[\cite{bak-sch23}, \S 2, (7)] \label{deligne-ext}
    Deligne's canonical extension lifts the isomorphism on \(B^\circ\) to
    \[T_{B^{sm}} \simeq \mathcal{V}^{>-1} \cap i_*(F^1 \mathcal{V}).\]
\end{lemma}

Now we are finally prepared to refine \Cref{mv-rep-classification}. While 
dimension-counting has narrowed down the types of isotypic decomposition which 
may occur to only three cases, it has failed to rule out the possibility
that \(V_\C\) may be of complex type or quaternionic type.

\begin{theorem} \label{mv-real-type}
    The monodromy representation associated to any maximal variation Lagrangian
    fibration is of real type.
\end{theorem}
\begin{proof}
Suppose that $V_\C$ were quaternionic, i.e. $V_\C = U^{\oplus 2}$, for $U$ complex
irreducible. Take \(\mathcal{U}\) to be the flat bundle on \(B^\circ\) associated to
\(U\); we naturally have \(\mathcal{V} = \mathcal{U}^{\oplus 2}\). Since 
this is a decomposition on the level of variations of Hodge structure, 
\(T_{B^\circ} \simeq F^1 \mathcal{V} = \left(F^1 \mathcal{U}\right)^{\oplus 2}\).
But now we may apply \Cref{deligne-ext} to get 
\begin{equation}
    T_{B^{sm}} \simeq 
    \left(\mathcal{U}^{>-1} \cap i_*(F^1 \mathcal{U})\right)^{\oplus 2}.
\end{equation}
In particular, the sheaf of reflexive 1-forms
\(j_* (T_{B^{sm}}^\vee) = \Omega_B^{[1]}\) on \(B\) is a square. Thus the space 
\(H^1(B, \Omega_B^{[1]})\) is even-dimensional. But by 
\cite[Corollary 2.3]{bak-lenh20}, \(H^{p,q}(B) \simeq H^q(X, \Omega_B^{[p]})\)
for \(p + q \leq 2\). 
Since \(H^0(B, \Omega_B^{[2]})\) contains the image of \(\overline{\sigma}\) under 
the identification given by Matsushita's theorem \cite{schnell23}, we know that 
\(H^2(B, \C)\) is non-empty. Thus it has dimension at least 2.
But \(B\) is known to have the same rational cohomology
as \(\mathbb{P}^n\) (see, for instance, \cite[Theorem 2.1]{huy22}). As 
\(H^2(\mathbb{P}^n, \Q) \simeq \Q\), this is a contradiction.

Now we show that $V_\C$ is not complex. Once again, suppose for contradiction 
that it were, let $V_\C = U \oplus \overline{U}$, for $U$ an irreducible
$\C$-local system. Then on the level of flat bundles, $\mathcal{V}$ splits as
$\mathcal{U}_1 \oplus \mathcal{U}_2$. Each $\mathcal{U}_i$ comes equipped with a
sub-bundle $F^1\mathcal{V}_i$, which respects the Hodge filtration
on $\mathcal{V}$, i.e.
\begin{equation} \label{eq-tangent-bundle-local}
    T_{B^\circ} \simeq F^1 \mathcal{V} = F^1 \mathcal{U}_1 \oplus F^1 \mathcal{U}_2.
\end{equation}
Now apply \Cref{deligne-ext} to get a splitting of \(T_{B^{sm}}\) as
\begin{equation}
    T_{B^{sm}} \simeq \left(\mathcal{U}_1^{>-1} \cap i_*(F^1 \mathcal{U}_1) \right)
    \oplus \left(\mathcal{U}_2^{>-1} \cap i_*(F^1 \mathcal{U}_2)\right)
    =: \mathcal{E}_1 \oplus \mathcal{E}_2.
\end{equation}
Note that both \(\mathcal{E}_1\) and \(\mathcal{E}_2\) are foliations. Their 
integrability follows from \cite[Lemma 3.3]{bak-sch23}. We will now show that 
at any general point of \(B^{sm}\), at least one of the two foliations admits an
algebraically integrable subfoliation. Such foliations are induced by morphisms 
\(\pi: B \to W\), which will produce a non-trivial \(\Q\)-subvariation 
of \(V\), contradicting \Cref{real-irreducibility}.

Note that \(\det (j_* \mathcal{E}_1) \otimes \det (j_* \mathcal{E}_2) = 
\omega_{B}^{\vee}\) is \(\Q\)-ample. By \cite[Theorem 1.1]{mats15}, \(B\) has 
Picard rank 1, so one of the two summands must therefore have ample determinant.
In other words, we may assume without loss of generality that \(j_* \mathcal{E}_1\)
is a \(\Q\)-Fano foliation. We will use the argument of
\cite[Corollary 2.12]{araujo18} to show that there exists
\(\mathcal{F} \subseteq j_* \mathcal{E}_1\) an algebraically integrable subfoliation.

Fix an ample class \(H\) on \(B\), and consider the Harder-Narasimhan filtration of
\(j_* \mathcal{E}_1\) induced by the corresponding slope \(\mu_H\):
\begin{equation} \label{eq-hn-filtration}
    0 \subsetneq \mathcal{F} = \mathcal{F}_1 \subsetneq \cdots \subsetneq 
    \mathcal{F}_k = j_*\mathcal{E}_1, \qquad \mu(\mathcal{F}) > 
    \mu(\mathcal{F}_2 / \mathcal{F}_1) > \cdots > 
    \mu(\mathcal{F}_{k} / \mathcal{F}_{k-1}).
\end{equation}
We claim that \(\mathcal{F}\) is a subfoliation of \(j_* \mathcal{E}_1\). Indeed, 
given a local section \(v \in \mathcal{F}(U)\), we may consider the image of the 
Lie bracket \([v, -]: \mathcal{F}(U) \to \mathcal{E}_1(U)\). This is not a morphism
of \(\mathcal{O}_B(U)\)-modules, since for any \(f \in \mathcal{O}_B(U)\) and 
\(w \in \mathcal{F}(U)\) we have
\begin{equation} \label{eq-leibniz}
    [v, fw] = \nabla_v(f)\cdot w + f[v, w].
\end{equation}
However, we may compose this map with the projection 
\(\mathcal{E}_1 \to \mathcal{E}_1 / \mathcal{F}\), in which case the first term of 
\eqref{eq-leibniz} dies. This composition is now \(\mathcal{O}_B\)-linear. Moreover,
since the choice of local section \(v\) was arbitrary, we in reality have a morphism 
of \(\mathcal{O}_B\)-modules \([-,-]: \mathcal{F}^{\otimes 2} \to 
\mathcal{E}_1 / \mathcal{F}\). Suppose this morphism was non-zero; then 
there would need to be some graded piece of the Harder-Narasimhan filtration on which
the Lie bracket had non-zero image, say 
\begin{equation} \label{eq-hn-tensor}
    \mathcal{F}^{\otimes 2} \to \mathcal{E}_1 / \mathcal{F} \to 
    \mathcal{F}_i / \mathcal{F}_{i-1}.
\end{equation}
By assumption, both \(\mathcal{F}\) and \(\mathcal{F}_i / \mathcal{F}_{i-1}\) are 
semistable with the slope of the former strictly greater than the slope of the latter.
But then there are no non-zero maps 
\(\mathcal{F} \to \mathcal{F}_i / \mathcal{F}_{i-1}\). But in characteristic zero,
the tensor product of stable sheaves is semistable; see, for instance, 
\cite[\S 2.1.2]{langer05}. Thus \(\mathcal{F}^{\otimes 2}\) is also semi-stable of slope
\(2 \mu_H(\mathcal{F})\). In particular,
it too admits no non-zero maps to \(\mathcal{F}_i / \mathcal{F}_{i-1}\),
proving \([\mathcal{F}, \mathcal{F}] \subseteq \mathcal{F}\).

Now construct a general complete intersection curve \(C \subseteq B\) avoiding 
the singular loci of \(B\) and \(\mathcal{F}\) as follows: for 
\(m_1, \dots, m_{n-1}\) sufficiently large integers, let \(H_i\) be a general
member of the complete linear series \(|m_i H|\), and set 
\(C := H_1 \cap \cdots \cap H_{n-1}\). By the Mehta-Ramanathan Theorem 
\cite{mehta-ram84}, the restriction of \eqref{eq-hn-filtration} to \(C\)
is still the Harder-Narasimhan filtration of \(j_* \mathcal{E}_1|_C\).
In particular, we see that every filtered piece \(\mathcal{F}_i|_C\) is locally free.
Since \(j_* \mathcal{E}_1\) was \(\Q\)-Fano, we have \(\mu_H(j_* \mathcal{E}_1) > 0\),
so it has positive degree on \(C\). Consequently, \(\mathcal{F}|_C\) also has positive
degree. By a theorem of Hartshorne \cite[Theorem 2.4]{hart71}, \(\mathcal{F}|_C\)
is ample if and only if every quotient bundle of \(\mathcal{F}|_C\) has positive
degree. But \(\mathcal{F}|_C\) is the maximal destabilizing subsheaf, so it admits 
no nontrivial quotients. Thus, by \Cref{bog-mcq-main-theorem}, 
every leaf of \(\mathcal{F}\) through \(C\) is algebraic. Now because \(C\) was general,
the general leaf of \(\mathcal{F}\) is algebraic,
or equivalently, \(\mathcal{F}\) is algebraically integrable.

Now, by \cite[Theorem 2.1.9]{chen23}, there exists a morphism \(\pi: B \to W\)
for \(W\) a smooth variety inducing
\(\mathcal{F}\), i.e. \(\mathcal{F} = \ker(d\pi)\). Consider the 
variation of \(\Q\)-Hodge structures \(\pi^*\pi_* V\) on \(B^\circ\). Fiberwise, 
\((\pi_* V)_w = H^0(B_w^\circ, V)\). Hence, by the theorem of the fixed part
\cite{schmid73}, \(\pi^*\pi_* V \subseteq V\) is a \(\Q\)-subvariation. Now, since \(V\) is 
irreducible by \Cref{real-irreducibility}, we must either have equality or vanishing.
If \(\pi^*\pi_* V = V\), the period map 
\(\mathcal{P}: B^\circ \to \leftQ{\mathcal{D}}{\Gamma}\) would factor through 
a Zariski open subset of \(W\). Since \(\dim W < \dim B\), this contradicts the 
maximal variation assumption.

Thus we must have \(\pi^* \pi_* V = 0\). But for a general point \(w \in W\) we have
the fiberwise containment
\begin{equation}
    (\pi_* \mathcal{V})_w = H^0(B_w^\circ, \mathcal{V}) \supseteq H^0(B_w^\circ, \mathcal{U}_2) \neq 0.
\end{equation}
The non-vanishing comes from the fact that \(\mathcal{U}_2\) is fixed along the leaves of 
\(\mathcal{F} \subseteq j_* \mathcal{E}_1\), by construction. 
\end{proof}

\begin{corollary} \label{complex-irreducibility}
    If \(V_\C\) is of maximal variation, it is irreducible as a \(\C\)-local system.
\end{corollary}
\begin{proof}
This follows immediately from \Cref{mv-real-type} and 
\Cref{mv-rep-classification}.
\end{proof}

\section{Isotrivial Fibrations}

To more cleanly deal with the case where \(f: X \to B\) is isotrivial, we introduce
a definition.

\begin{definition}
    Let \(\Gamma\) be a finitely generated group, \(K \subseteq \C\) a field.
    A \(K\)-Hodge representation of \(\Gamma\) is the data of a 
    polarized \(K\)-Hodge structure \((V, \bigoplus_{p+q=w} V_\C^{p,q}, q)\) and 
    a representation \(\rho: \Gamma \to \GL(V, q)\) acting by Hodge automorphisms.
\end{definition}

The notion of Hodge representation is useful precisely because it captures the 
data of isotrivial variations of Hodge structure. In particular,
since \(f\) is isotrivial, we may fix a basepoint \(b \in B^\circ\) and note 
that \(V_\Q\) has the structure of a weight 1 Hodge representation 
\(((V_\Q)_b, (V_\C)_b^{1,0} \oplus (V_\C)_b^{0,1}, q, \rho)\)
with respect to the monodromy representation
\[\rho: \Gamma = \pi_1(B^\circ, b) \to GL(V_\Q, q).\]
The choice of the basepoint does not matter; such a Hodge representation 
completely captures the data of \(V_\Q\). Hence, we will abuse notation by 
dropping the subscript, identifying \(V_\Q\) as a \(\Q\)-Hodge representaion
of \(\pi_1(B^\circ)\).

Now we recover the stucture theorem of Kim, Laza, and Martin.
\begin{theorem}[\cite{kim-laza-martin23}, Theorem 1.3] \label{elliptic-curve-theorem}
    There exists an elliptic curve \(E\) along with an isomorphism of \(\Q\)-Hodge 
    structures \(H^1(E, \Q)^{\oplus n} \to V_\Q\).
\end{theorem}
\begin{remark}
Note that we do not claim this map to be an isomorphism of Hodge 
representations; indeed, we will see that \(H^1(E, \Q)^{\oplus n}\) only admits 
a compatible Hodge representation strucure  
locally. Phrased in terms of isotrivial variations,
this is an identification of Hodge structures between a general fiber of 
\(V_\Q\) and \(H^1(E, \Q)^{\oplus n}\). Equivalently, the smooth fibers of
\(f: X \to B\) are isogeneous to \(E^n\).
\end{remark}
\begin{proof}
Let \(D = B^{sm} \smallsetminus B^{\circ}\). We consider the local monodromy of 
\(V_\Q\) in a small tubular neighborhood \(U \subseteq B^{sm}\) of \(D\) where
\(D \cap U\) is smooth. Set \(U^\circ = U \cap B^\circ\), and choose a basepoint
\(b \in U^\circ\). Let \(T_b\) denote the matrix representing the local monodromy
operator in \(U^\circ\). 
By \cite[Lemma 3.6]{bak-sch23}, \(T_b\) has only two
non-trivial conjugate eigenvalues \(\lambda, \overline{\lambda}\) with
\begin{equation} \label{eq-possible-eigenvalues}
    (\lambda, \overline{\lambda}) \in \{(\zeta_6, \overline{\zeta}_6),
    (\zeta_4, \overline{\zeta}_4), (\zeta_3, \overline{\zeta}_3), (-1, -1)\}.
\end{equation}
We have a natural Hodge substructure \((V_\Q)^{T_b} \subseteq V_\Q\) 
given by the \(1\)-eigenspace of \(T_b\); denote by \(H\) its orthogonal 
complement with respect to the Hermitian form on \(V_\Q\). \(H\) too 
is a Hodge substructure, and because \(H\) is two-dimensional and weight one,
it is irreducible. Notice that there exists an isotrivial family of elliptic 
curves \(\mathcal{E} \to U^\circ\) with fiber \(E\), such that we have an 
isomorphism of \(\pi_1(U^\circ)\)-Hodge representations \(H^1(E, \Q) \simeq H\).

There are no Hodge morphisms between different isotypic 
components of a Hodge structure, so \(V_\Q\) has multiple isotypic components 
as a Hodge structure if and only if it has multiple isotypic components as a 
Hodge representation. But by \Cref{real-irreducibility}, \(V_\Q\) is
irreducible as a Hodge representation, so it has only a single isotypic component 
as a Hodge structure. Since \(H \subseteq V_\Q\) is irreducible, we must get 
an isomorphism of \(\Q\)-Hodge structures \(V_\Q \simeq H^{\oplus n}\).
\end{proof}

Next we seek to understand the structure of \(V_\C\) by applying 
\Cref{rep-classification}. As in \Cref{mv-rep-classification}, we will see that 
in each of the three cases, the dimensions of \(B\) and \(U\) are determined.
Moreover, due to the previous theorem, we will see that the isotypic decomposition 
of \(V_\C\) actually occurs over a degree two extension of \(\Q\).

\begin{proposition} \label{isotrivial-rep-classification}
    Let \(E\) be the elliptic curve from above, \(H = H^1(E, \Q)\).
    \begin{itemize}
    \item[(a)] \(V_\C\) is of real type if and only if we have an isomorphism
    of \(\Q\)-Hodge representations \(V_\Q = H \otimes W\) for \(W\) a
    \(\Q\)-local system irreducible over \(\C\).
    \item[(b)] Let \(K = \mathrm{End}_{\Q HS}(H)\) be the CM field of \(E\).
    \(V_\C\) is of complex type precisely when \(K \neq \Q\) and there is an 
    isomorphism of \(K\)-Hodge representations  \(V_K \simeq V_K^{1,0} \oplus V_K^{0,1}\).
    Here \(V_K^{1,0}\) and \(V_K^{0,1}\) are non-isomorphic \(K\)-local systems which
    are irreducible over \(\C\).
    \item[(c)] \(V_\C\) is not quaternionic.
    \end{itemize}
\end{proposition}

\begin{proof}
If \(V_\C\) is real or quaternionic, it must be
of the form \(V_\C = B \otimes U\), where \(U\) is an irreducible \(\C\)-local system.
However, since \(V_\C\) is isotrivial, \(U\) must have level zero. Since 
\(V_\C = V^{1,0} \oplus V^{0,1}\), \(B\) must have level one. In particular, it is 
at least two-dimensional.
We first claim that \(V_\C\) cannot be quaternionic. Indeed, if it were, then
by \Cref{proof-of-rep-classification}, \(h^0(\bigwedge^2 U)\) would contain
an alternating form. Consider the containment
\(\bigwedge^2 U \oplus \bigwedge^2 U \subseteq \bigwedge^2 V_\C\).
We know by Deligne's global invariant cycle theorem that 
\(h^0 \left(\bigwedge^2 V_\C\right) = 1\), so this containment is impossible, a contradiction.

Next, suppose that \(V_\C = B \otimes W_\C\) has real type. We claim
that the dimension of \(B\) is exactly two. Consider once again the containments
\begin{equation}
    \bigwedge^2 V_\C \supseteq (W_\C^{\otimes 2})^{\oplus \binom{\dim B}{2}}
    \supseteq (\text{Sym}^{2}W_\C)^{\oplus \binom{\dim B}{2}}
\end{equation}
Again by the global invariant cycle theorem, we know that the 
\(h^0(\bigwedge^2 V_\C) = 1\). But by \Cref{proof-of-rep-classification},
\(h^0(\text{Sym}^2 W_\C) = 1\). This forces \(\dim B = 2\).

Otherwise, suppose that \(V_\C\) is of complex type, 
\(V_\C = (B \otimes U) \oplus (\overline{B} \otimes \overline{U})\).
Then \(B\) and \(\overline{B}\) both have level zero. But now they must both be 
one-dimensional. If not, any subspace \(C \subseteq B\) is a Hodge subvariation,
with \((C \otimes U) \oplus (\overline{C} \otimes \overline{U}) \subsetneq V_\C\)
defined over \(\R\). This contradicts the irreducibility of \(V_\R\) as a Hodge 
structure. Therefore \(V_\C \simeq U \oplus \overline{U}\).

So far, we have not used the Hodge substructure \(H \subseteq V_\Q\). Consider 
the natural evaluation map 
\begin{equation}
    H \otimes_\Q \Hom_{\Q HS}(H, V_\Q) \to V_\Q.
\end{equation}
This map is always surjective. Suppose that \(E\) is not CM, i.e.
\(\End_{\Q HS}(H) = \Q\). Then by the previous theorem, 
\(\dim H \otimes_\Q \Hom_{\Q HS}(H, H^{\oplus n}) = 2n\), so that 
this map is an isomorphism of \(\Q\)-Hodge structures. Now notice that we may 
equip the domain with a Hodge representation, where \(\pi_1(B^\circ)\) acts 
trivially on \(H\) and \(\gamma . \varphi: h \to \varphi(\gamma h)\) for 
\(h \in H\) and \(\varphi \in \Hom_{\Q HS}(H, V_\Q)\). This extends the evaluation 
map to an isomorphism of Hodge representations. In particular, \(V_\C\) has real 
type.

Now suppose that \(E\) does have CM, \(\End_{\Q HS}(H) = K\) is an imaginary 
quadratic extension of \(\Q\). Since \(H_K = H_K^{1,0} \oplus H_K^{0,1}\) as a 
Hodge structure, the previous theorem gives us an analogous splitting of 
Hodge structures \(V_K = V_K^{1,0} \oplus V_K^{0,1}\). Moreover, Hodge automorphisms 
preserve the Hodge grading, so this gives a splitting of the Hodge representation 
\(V_K = V_K^{1,0} \oplus V_K^{0,1}\) into \(K\)-local systems. If they are 
non-isomorphic, then \(V_\C = V_\C^{1,0} \oplus V_\C^{0,1}\) is a splitting of 
\(V_\C\) into non-isomorphic \(\C\)-local systems. As we showed above, 
this is the complex type case, so that \(V_\C^{1,0}\) and \(V_\C^{0,1}\) are 
irreducible \(\C\)-local systems. 

Alternatively, suppose that \(V_K^{1,0} \simeq V_K^{0,1}\), so that 
\(V_K \simeq B \otimes U\) for \(U\) an irreducible \(K\)-local system.
If \(U = W_K\) for \(W\) a \(\Q\)-local system, then we are again in the real type
case. Thus \(B\) is two-dimensional and level one, so it is again identified 
with the Hodge substructure \(H \subseteq V_\Q\), and \(V_\Q = H \otimes W\).
Finally, suppose \(U\) is an irreducible \(K\)-local system which is not defined
over \(\Q\). Then \(U_\C\) is a \(\C\)-local system which is not defined 
over \(\R\). As we have seen, it cannot be quaternionic, so it must be of 
complex type. But if \(U_\C \neq \overline{U}_\C\), then \(U \neq \overline{U}\).
This contradicts our assumption.
\end{proof}

\begin{remark}
One may have expected that \(V_\C\) has complex type if and only if \(E\) is CM. 
The reason why it is not sufficient for \(E\) to be CM is because the monodromy 
may never actually make use of the CM structure. Indeed, \(V_\C\) has complex type
precisely when there is a local generator \(\gamma \in \pi_1(B^\circ)\) which
acts on \(H\) by a CM automorphism.
\end{remark}

There is an alternative perspective one may take in classifying isotrivial Lagrangian
fibrations. Given $f: X \to B$, one may ask for the intermediate trivialization
$\pi: C^\circ \to B^\circ$, the smallest cover of $B^\circ$ trivializing
$\pi^* R^1 f_*\Q_{X^\circ}$. We may then form the pullback
diagram
\begin{equation*}
\begin{tikzcd}
	Z^\circ \arrow{r} \arrow{d} & X^\circ \arrow{d}{f} \\
	C^\circ \arrow{r}{\pi} & B^\circ
\end{tikzcd}
\end{equation*}

By \cite{kim-laza-martin23}, $C^\circ$ must compactify to a variety which is either
of general type or an Abelian torsor.
In fact, if we assume that the singular fibers of \(f\) are not multiple, 
\cite{kim-laza-martin23} show that the singular fibers of any Lagrangian
fibration follow the Kodaira classification of elliptic surface singularities.

\begin{proposition}[\cite{kim-laza-martin23}, Proposition 4.20] \label{klm-kodaira-type}
    Suppose that \(f: X \to B\) is an isotrivial Lagrangian fibration whose 
    singular fibers are non-multiple. A general singular fiber of \(f\) can be
    of Kodaira type \(\mathrm{II}\), \(\mathrm{III}\), \(\mathrm{IV}\), \(\mathrm{I}_0^*\),
    \(\mathrm{II}^*\), \(\mathrm{III}^*\), or \(\mathrm{IV}^*\).
\end{proposition}

We exibit some examples of isotrival elliptic K3 surfaces
to show how the classification of \cite{kim-laza-martin23} relates to our own.
We owe these examples in part to \cite{sawon14}.

\begin{example}
Let \(f: X \to \mathbb{P}^1\) be an elliptic K3 surface, and suppose
that \(X\) has singular fibers of Kodaira type II (cuspidal cubic curve). Each one
has Euler characteristic 2, so there are twelve of them; around each of the twelve
punctures \(b_i \in \mathbb{P}^1\), the monodromy matrix can be locally expressed as
\begin{equation}
    T_{b_i} = \begin{pmatrix}
        1 & 1 \\ -1 & 0
    \end{pmatrix}.
\end{equation}
Since the matrix is of order six, we may construct a degree six cover 
\(\pi: C^\circ \to B^\circ\) trivializing \(R^1 f_* \Z_{X^\circ}\).
But \(C^\circ\) then extends to a ramified cover \(\pi: C \to \mathbb{P}^1\).
By Riemann-Hurwitz, 
\begin{equation}
    \chi(C) = 6 \cdot 2 - (6-1) \cdot 12 = -48
\end{equation}
so \(C\) has genus twenty-five. In particular, it is of general type. The monodromy matrix
acts on a general fiber \(E\) by complex multiplication, so the associated
monodromy representation is of complex type.
\end{example}

\begin{example} \label{example-Kummer-real}
Take \(E \times E \to E\) to be projection onto a factor for \(E\) an elliptic curve, 
and quotient by
the involution. After blowing up the sixteen fixed points, we get a Lagrangian 
fibration \(f: X \to \mathbb{P}^1\). Since the map \(E \to \mathbb{P}^1\) has
four branch points, we see that \(f\) has four singular fibers, each of which is
of Kodaira type \(I_0^*\). The associated Dynkin diagram is \(\widetilde{D_4}\),
with fibers consisting of four \(\mathbb{P}^1\) meeting a doubled 
\(\rightQ{E}{\pm 1} = \mathbb{P}^1\)
at the four branch points. The local monodromy matrix over each fiber is given by
\begin{equation}
    T_b = \begin{pmatrix}
        -1 & 0 \\ 0 & -1
    \end{pmatrix}.
\end{equation}
This lattice automorphism is not CM, so the local system is of real type.
The associated covering map is \(E \to \mathbb{P}^1\) of degree 2, where
\(E\) is Abelian.
\end{example}

\begin{example} \label{example-Kummer-complex-1}
Take \(E \times E \to E\) as in \Cref{example-Kummer-real}, but this time we
instead fix \(E = \rightQ{\C}{\left<1, \zeta\right>}\) for \(\zeta\) a primitive
sixth root of unity. Quotienting by the diagonal action of the group
\begin{equation}
    G = \left<\begin{pmatrix}
        \zeta^2 & 0 \\ 0 & \zeta^{-2}
    \end{pmatrix}\right>
\end{equation}
and blowing up the nine fixed points gives us a Lagrangian fibration 
\(f: X \to \mathbb{P}^1\). The three singular fibers above each of the fixed points
of \(\left<\zeta^2\right>\) acting on \(E\) consist of \(E_6\) singularities of Kodaira
type \(\text{IV}^*\). In particular, the local monodromy matrices all have the form
\begin{equation*}
    T_b = \begin{pmatrix}
        -1 & -1 \\ 1 & 0
    \end{pmatrix}.
\end{equation*}
They have order three, so in particular, the local system is of complex type. \(B^\circ\) 
is covered by a degree three map from \(E\), so the fibration is of Abelian type.
\end{example}

\begin{example}
In \Cref{example-Kummer-complex-1}, replace \(E\) by 
\(\rightQ{\C}{\left<1, i\right>}\) and \(G\) by
\(\left<\begin{pmatrix}
    i & 0 \\ 0 & -i
\end{pmatrix}\right>\).
We once again produce a Lagrangian fibration
\begin{equation}
    f: X = \leftQ{E \times E}{G} \longrightarrow 
    \leftQ{E}{\left<i\right>} \simeq \mathbb{P}^1.
\end{equation}
The action of \(G\) on \(E\) has three fixed points, giving three singular fibers.
Two are of Kodaira type \(\text{III}^*\) with Dynkin diagram \(\widetilde{E}_7\), 
while the last is forced to be \(\text{I}_0^*\) as in \Cref{example-Kummer-real}.
The new fibers have local monodromy
\begin{equation}
    T_b = \begin{pmatrix}
        0 & -1 \\ 1 & 0
    \end{pmatrix}
\end{equation}
of order four, so it is CM. Once again, \(\left(\mathbb{P}^1\right)^\circ\) is 
covered and compactified to \(E\), so it has Abelian type.
\end{example}

Notice that while we managed to produce examples of Abelian type fibrations which 
were of either representation type, we do not have an example of a general type 
fibration which is of real type. In fact, we claim that this is the case in higher 
dimensions as well.

\begin{proposition} \label{application-isotrivial-real}
    \(V_\C\) is of complex type if and only if it splits over \(\Q(i)\) or 
    \(\Q(\sqrt{-3})\). If \(V_\C\) has real type, then the intermediate
    trivialization \(C^\circ\) must compactify to an Abelian torsor.
\end{proposition}
\begin{proof}
    The local monodromy group around any singular fiber is 
    \(\mu_n\) for \(n=2,3,4,6\). If \(V_\C\) is to be of real type, we must have
    \(n=2\), which is attained only when the singular fiber is of Kodaira type 
    \(\text{I}_0^*\). But now the first assertion follows directly from 
    \Cref{isotrivial-rep-classification} and \Cref{klm-kodaira-type}. The second comes
    from \cite[Proposition 2.24]{kim-laza-martin23}, which states that \(C^\circ\)
    compactifies to an Abelian torsor if and only if the general singular fibers of
    \(f\) are of Kodaira type \(\text{I}_0^*\), \(\text{II}^*\), \(\text{III}^*\), or 
    \(\text{IV}^*\), so that all real-type \(V_\C\) are trivialized by an Abelian
    torsor.
\end{proof}

\bibliographystyle{alpha}
\bibliography{refs}

@article{huy22,
  title={Lagrangian fibrations},
  author={Huybrechts, Daniel and Mauri, Mirko},
  journal={Milan Journal of Mathematics},
  volume={90},
  number={2},
  pages={459--483},
  year={2022},
  publisher={Springer}
}

@article{bak22,
  title={A short proof of a conjecture of {M}atsushita},
  author={Bakker, Benjamin},
  journal={Advances in Mathematics},
  volume={482},
  pages={110554},
  year={2025},
  publisher={Elsevier}
}

@misc{bak-sch23,
      title={A {H}odge-theoretic proof of {H}wang's theorem on base manifolds of {L}agrangian fibrations}, 
      author={Benjamin Bakker and Christian Schnell},
      year={2023},
      eprint={2311.08977},
      archivePrefix={arXiv},
      primaryClass={math.AG},
      url={https://arxiv.org/abs/2311.08977}, 
}

@article{del87,
  title={A finiteness theorem for monodromy},
  author={Deligne, Pierre},
  journal={Discrete Groups in Geometry and Analysis: Papers in Honor of GD Mostow on His Sixtieth Birthday},
  pages={1--19},
  year={1987},
  publisher={Springer}
}

@incollection{mats16,
  title={On deformations of {L}agrangian fibrations},
  author={Matsushita, Daisuke},
  booktitle={K3 surfaces and their moduli},
  pages={237--243},
  year={2016},
  publisher={Springer}
}

@inproceedings{vois18,
  title={Torsion points of sections of {L}agrangian torus fibrations and the {C}how ring of hyper-{K}{\"a}hler manifolds},
  author={Voisin, Claire},
  booktitle={Geometry of Moduli 14},
  pages={295--326},
  year={2018},
  organization={Springer}
}

@article{del71,
  title={Hodge theory. {II}},
  author={Deligne, Pierre},
  journal={Publication Mathematique IHES},
  volume={40},
  pages={5--58},
  year={1971}
}

@article{hwa08,
  title={Base manifolds for fibrations of projective irreducible symplectic manifolds},
  author={Hwang, Jun-Muk},
  journal={Inventiones mathematicae},
  volume={174},
  number={3},
  pages={625--644},
  year={2008},
  publisher={Springer}
}

@article{deb22,
  title={Hyper-k{\"a}hler manifolds},
  author={Debarre, Olivier},
  journal={Milan Journal of Mathematics},
  volume={90},
  number={2},
  pages={305--387},
  year={2022},
  publisher={Springer}
}

@article{mats15,
  title={On base manifolds of {L}agrangian fibrations},
  author={Matsushita, Daisuke},
  journal={Science China Mathematics},
  volume={58},
  pages={531--542},
  year={2015},
  publisher={Springer}
}

@article{mats05,
  title={Higher direct images of dualizing sheaves of {L}agrangian fibrations},
  author={Matsushita, Daisuke},
  journal={American Journal of Mathematics},
  volume={127},
  number={2},
  pages={243--259},
  year={2005},
  publisher={Johns Hopkins University Press}
}

@book{bog-mcq16,
  title={Rational curves on foliated varieties},
  author={Bogomolov, Fedor and McQuillan, Michael},
  year={2016},
  publisher={Springer}
}

@article{kim-laza-martin23,
  title={Isotrivial {L}agrangian fibrations of compact hyper-{K}{\"a}hler manifolds},
  author={Kim, Yoon-Joo and Laza, Radu and Martin, Olivier},
  journal={Journal de Math{\'e}matiques Pures et Appliqu{\'e}es},
  pages={103810},
  year={2025},
  publisher={Elsevier}
}

@misc{sawon14,
      title={Isotrivial elliptic {K}3 surfaces and {L}agrangian fibrations}, 
      author={Justin Sawon},
      year={2014},
      eprint={1406.1233},
      archivePrefix={arXiv},
      primaryClass={math.AG},
      url={https://arxiv.org/abs/1406.1233}, 
}

@article{schmid73,
  title={Variation of {H}odge structure: the singularities of the period mapping},
  author={Schmid, Wilfried},
  journal={Inventiones mathematicae},
  volume={22},
  number={3},
  pages={211--319},
  year={1973},
  publisher={Springer-Verlag Berlin/Heidelberg}
}

@inproceedings{araujo18,
  title={Positivity and algebraic integrability of holomorphic foliations},
  author={Araujo, Carolina},
  booktitle={Proceedings of the International Congress of Mathematicians: Rio de Janeiro 2018},
  pages={547--563},
  year={2018},
  organization={World Scientific}
}

@misc{schnell23,
      title={Hodge theory and {L}agrangian fibrations on holomorphic symplectic manifolds}, 
      author={Christian Schnell},
      year={2023},
      eprint={2303.05364},
      archivePrefix={arXiv},
      primaryClass={math.AG},
      url={https://arxiv.org/abs/2303.05364}, 
}

@article{bak-lenh20,
  title={A global {T}orelli theorem for singular symplectic varieties},
  author={Bakker, Benjamin and Lehn, Christian},
  journal={Journal of the European Mathematical Society},
  volume={23},
  number={3},
  pages={949--994},
  year={2020}
}

@article{hart71,
  title={Ample vector bundles on curves},
  author={Hartshorne, Robin},
  journal={Nagoya Mathematical Journal},
  volume={43},
  pages={73--89},
  year={1971},
  publisher={Cambridge University Press}
}

@article{greb-lenh14,
  title={Base manifolds for {L}agrangian fibrations on hyperk{\"a}hler manifolds},
  author={Greb, Daniel and Lehn, Christian},
  journal={International Mathematics Research Notices},
  volume={2014},
  number={19},
  pages={5483--5487},
  year={2014},
  publisher={Oxford University Press}
}

@article{mats99,
  title={On fibre space structures of a projective irreducible symplectic manifold},
  author={Matsushita, Daisuke},
  journal={Topology},
  volume={38},
  number={1},
  pages={79--83},
  year={1999},
  publisher={Elsevier}
}

@article{huy-xu22,
  title={Lagrangian fibrations of hyperk{\"a}hler fourfolds},
  author={Huybrechts, Daniel and Xu, Chenyang},
  journal={Journal of the Institute of Mathematics of Jussieu},
  volume={21},
  number={3},
  pages={921--932},
  year={2022},
  publisher={Cambridge University Press}
}

@article{langer05,
  title={Moduli spaces of sheaves and principal {G}-bundles},
  author={Langer, Adrian},
  journal={Algebraic geometry—Seattle},
  volume={Part 1},
  pages={273--308},
  year={2005}
}

@article{mehta-ram84,
  title={Restriction of stable sheaves and representations of the fundamental group},
  author={Mehta, Vikram Bhagvandas and Ramanathan, Annamalai},
  journal={Inventiones mathematicae},
  volume={77},
  number={1},
  pages={163--172},
  year={1984},
  publisher={Springer}
}

@article{geeman-voisin16,
  title={On a conjecture of {M}atsushita},
  author={Geemen, Bert van and Voisin, Claire},
  journal={International Mathematics Research Notices},
  volume={2016},
  number={10},
  pages={3111--3123},
  year={2016},
  publisher={Oxford University Press}
}

@misc{chen23,
      title={Minimal model program for algebraically integrable foliations and generalized pairs}, 
      author={Guodu Chen and Jingjun Han and Jihao Liu and Lingyao Xie},
      year={2023},
      eprint={2309.15823},
      archivePrefix={arXiv},
      primaryClass={math.AG},
      url={https://arxiv.org/abs/2309.15823}, 
}

\end{document}